\documentclass{amsart}
\usepackage{a4wide,amssymb,color}

\newtheorem{theorem}{Theorem}[section]
\newtheorem{proposition}[theorem]{Proposition}
\newtheorem{corollary}[theorem]{Corollary}
\newtheorem{example}[theorem]{Example}

\theoremstyle{definition}
\newtheorem{problem}[theorem]{Problem}
\newtheorem{definition}[theorem]{Definition}

\newtheorem{claim}[theorem]{Claim}

\newcommand{\defeq}{:=}
\newcommand{\IR}{\mathbb R}
\newcommand{\IC}{\mathbb C}
\newcommand{\IT}{\mathbb T}
\newcommand{\IZ}{\mathbb Z}
\newcommand{\IN}{\mathbb N}
\newcommand{\w}{\omega}
\newcommand{\e}{\varepsilon}
\newcommand{\Tau}{\mathcal T}
\newcommand{\cov}{\mathrm{cov}}
\newcommand{\M}{\mathcal M}
\newcommand{\A}{\mathcal A}
\newcommand{\Ra}{\Rightarrow}

\title{Detecting the real line among one-parametric topological groups}
\author{Taras Banakh, Kateryna Makarova, Oles Mazurenko}
\address{Ivan Franko National University of Lviv}
\email{t.o.banakh@gmail.com, kateryna.makarova@lnu.edu.ua, oles.mazurenko@lnu.edu.ua}

\begin{document}
	\begin{abstract} We characterize monothetic topological groups among one-parametric topological groups. In particular, we prove that a topological group of weight $<\cov(\M)$ is monothetic if and only if it is not isomorphic to the real line. This implies that the real line is a unique non-monothetic metrizable one-parametric topological group. Also we prove that the real line with the Bohr topology as a unique non-monothetic totally bounded $\IN$-scaled topological group.% in which all homomorphisms $ one-parametric topological group topological group $G$ is isomorphic to the real line if and only if it is a one-parameteric, metrizable, and not monothetic. This result is used in \cite{BM2} to prove that one-parametric groups in strictly convex metric group all are topologically isomorphic to the real line. The example of the Bohr topology on the real line demonstrates that metrizability is an essential assumption in our first claim. This motivates further study to characterize the Bohr topology on the real line as well as detect  monothetic one-parametric topologcal groups in a non-metrizable setting. Both issues are addressed and resolved in the present paper.
\end{abstract}
\maketitle

\section{Introduction}
The real line plays a fundamental role in mathematics. It serves as a cornerstone in numerous areas of scientific research. The main reason for that is the ability to endow the real line with various mathematical structures. Hence, researching the properties of $\IR$ inside a particular category of such mathematical structures could be helpful for more advanced studies. In particular, characterizing the real line as an object of some category is a common way to conduct such research. One can find helpful characterizations in the categories of ordered fields, topological spaces, groups etc. In this paper, we obtain a characterization of the real line in the category of topological groups, which turned out to be helpful in the other research \cite{BM2} to prove that one-parametric subgroups in any strictly convex metric group are isomorphic copies of the additive group of reals.

We prove that the real line is the unique non-monothetic metrizable one-parametric topological group. Also we detect monothetic groups among one-parametric topological group, and characterize the real line endowed with the Bohr topology and a unique non-monothetic totally bounded $\IN$-scaled topological group.
\smallskip

\section{Preliminaries}
Our study mainly considers the structure of a topological group. Let us recall its definition together with the definitions of several related fundamental properties.
	\begin{definition}
		A \emph{group} is an algebraic structure $(G, +, 0)$, consisting of a set $G$, a binary operation $+: G \times G \to G$  and an identity element $0$, satisfying the following axioms:
		\begin{enumerate}
			\item $\forall x,y,z\in G\;\;(x + y)+z = x+(y+z)$, \hfill (associativity)
			\item $\forall x\in G\;\;x+0 = x=0+x$, \hfill (identity)
			\item $\forall x\in G\;\exists y\in G\;\;x+y=0=y+x$. \hfill (inverse)
		\end{enumerate} 
	\end{definition}
	We will also use the operation $\cdot: \IZ \times G \to G$, $\cdot: (n,x) \mapsto nx$, which is naturally defined on the additive group $(G, +, 0)$ by the following recursive formulas: $0\cdot x=0$, $(n+1)\cdot x=n\cdot x+x$, and $-(n+1)\cdot x = -n \cdot x - x$ for all $n\in\IN\cup\{0\}$. If the group $G$ is abelian, the defined operation $\cdot$ turns it into a $\IZ$-module.
	\begin{definition}
		A \emph{topological group} is a group $(G, +, 0)$ equipped with a topology $\mathcal{T}$ such that the following maps are continuous:
		\begin{enumerate}
		\item the addition (group operation)
		\[
		+ : G \times G \to G, \quad+: (x, y) \mapsto x + y,
		\]
		\item the inversion
		\[
		i : G \to G, \quad i:x \mapsto -x.
		\]
		\end{enumerate}
	\end{definition}
	Unless stated otherwise, all topological groups are assumed to be Hausdorff. For a topological group $(G,\Tau)$, we denote by $\mathcal{T}_0 := \{U \in \mathcal{T}: 0 \in U\}$ the family of open neighborhoods of the identity element $0$ of the topological group. By $\IR$ we denote the additive group of reals equipped with the Euclidean topology $\mathcal{T}_\varepsilon$.
	\begin{definition}
		A map $\phi: S \to G$ between topological groups is called a {\em topological group isomorphism} if it is simultaneously a group homomorphism and a homeomorphism.
	\end{definition}
	If there exists a topological group isomorphism between topological groups $G$ and $H$, we will say that $G$ is {\em isomorphic} to $H$ and write $G \cong H$.

	\begin{definition}
		A topological group $G$ is called \emph{one-parametric} if there exists a surjective continuous group homomorphism $\phi: \IR \to G$ (called a {\em parametrization}).
	\end{definition}
	
	\begin{definition}
		A topological group $(G,+,0,\Tau)$ is called \emph{monothetic} if it contains a dense cyclic subgroup; that is, there exists an element $g \in G$ (called a {\em topological generator}) of $G$ such that $\overline{ \langle g \rangle} = \overline{\{ng: n \in \IZ\}} = G$.
	\end{definition}
	
\begin{definition}
A topological group $(G, +, 0, \mathcal{T})$ is called {\em totally  bounded} if for every open neighborhood $U \in \mathcal{T}_0$ of the identity $0 \in G$, there exists a finite subset $F \subseteq G$ such that
		$$G = \bigcup_{x \in F} (x + U).$$
\end{definition}

For ease of notation, we occasionally use property names to refer to the group topology rather than the topological group (for example, ‘monothetic group topology’). We will use these expressions interchangeably when no confusion can arise.
%	\begin{definition}
%		The {\em projective topology} on a set $X$ induced by the family of maps $f_\alpha: X \to Y_\alpha$ between the set $X$ and the topological spaces $(Y_\alpha, \mathcal{T}_\alpha$) is the coarsest topology on $X$ that makes all $f_\alpha$ continuous. It is generated by the sub-base
%		$$\bigcup_\alpha \{f_\alpha^{-1}[U_\alpha]: U_\alpha \in \mathcal{T}_\alpha\}.$$
%	\end{definition}
%	By the {\em pullback topology} of a one-parametric topological group $G$ we mean the projective topology on $\IR$ induced by the parametrization $\phi: \IR \to G$ of $G$.
	
\section{Uniqueness of Parametrization}
Let $(G, +, 0, \mathcal{T})$ be a one-parametric topological group. The following proposition shows that the parametrization $\phi: \IR \to G$ of $G$ is unique in a certain sense, which justifies why the term ``one-parametric group" is applied to $G$ itself rather than to a specific parametrization.
\begin{proposition}
	If $\phi: \IR \to G$ and $\psi: \IR \to G$ are two parametrizations of a (non-trivial) one-parametric topological group $G$, then there exists a (unique) topological group isomorphism $h: \IR \to \IR$ such that $\phi \circ h = \psi$.
\end{proposition}
\begin{proof}
	Observe that $\operatorname{ker}\phi = \{a \in \IR: \phi(a) = 0\}$ is a closed subgroup of $\IR$, which are known to be either $\{0\}$, $\IR$, or $a\IZ$ for some $a \in \IR_{>0}$ (see, e.g., \cite{Morris}, p.~27). 
	
	If $\operatorname{ker}\phi = \IR$, then the surjective homomorphism $\phi$ is trivial, and thus $G = \{0\}$. Hence, any topological group isomorphism $h: \IR \to \IR$ is as required.
	
	If $\operatorname{ker}\phi = a\IZ$, then the topological group $G$ can be identified with the circle group $\IR/a\IZ \cong \mathbb{T}\defeq\{z\in \IC:|z|=1\}$. Hence both $\phi$ and $\psi$ are continuous characters of $\IR$. Thus $\phi(x) = e^{2\pi i \frac{x}{a}}$ and $\psi(x) = e^{2\pi i \frac{x}{b}}$, where $\operatorname{ker}\psi= b\IZ.$ For the identity $\phi \circ h = \psi$ to hold, the map $h: \IR \to \IR$ must satisfy $e^{2\pi i \frac{x}{b}} = e^{2\pi i \frac{h(x)}{a}}$ for all $x\in\IR$. This is equivalent to $h(x) = \frac{a}{b}x + na$ for $n \in \IZ$. Such map is a topological group isomorphism if and only if $n = 0$.  Hence $h: \IR \to \IR$, $h: x \mapsto \frac{a}{b} x$ is a required unique topological group isomorphism.
	
	If $\operatorname{ker}\phi = \{0\}$, then $G$ is not isomorphic to $\mathbb{T}$, and therefore $\operatorname{ker}\psi = \{0\}$ as well. We claim that the map $h: \IR \to \IR$, $h = \phi^{-1} \circ \psi$, is the required unique topological group isomorphism. Observe that $h$ is already a group isomorphism; it remains to show that $h$ is a homeomorphism. 
	
	Write $\IR = \bigcup_{i \in \IN} I_i$ as the union of closed intervals $I_i$. Set $J_i := \phi[I_i]$ and $K_i := \psi[I_i]$. For each $i \in \IN$, the sets $J_i, K_i$ are compact subsets of $G$ and $G = \bigcup_{j, k \in \IN} J_j \cap K_k$. Since $\psi \restriction_{I_k}$ is a homeomorphism, we obtain $\IR =  \bigcup_{j, k \in \IN} \psi^{-1}[J_j \cap K_k]$ and each $\psi^{-1}[J_j \cap K_k]$ is compact. By the Baire Category Theorem there exist indices $j,k \in \IN$ such that $D := \psi^{-1}[J_j \cap K_k] \subseteq I_k$ has non-empty interior. Since $\phi \restriction_{I_j}$ and $\psi \restriction_{I_k}$ are homeomorphisms, the map $h$ is a homeomorphism on the non-empty open set $\operatorname{Int}(D)$. Since $h$ is a group isomorphism, being homeomorphic on a non-empty open set implies that it is a homeomorphism on all of $\IR$. Thus $h$ is indeed a topological  group isomorphism. Finally, uniqueness follows immediately from the fact that $\psi$ and $ \phi$ are bijective. 
\end{proof}

\section{Characterizations of the real line among one-parametric topological groups}\label{s:4}

In this section we prove two characterizations of the real line in the class of all one-parametric topological groups.

\begin{proposition}\label{p:1-par-bounded=>R} Let $G$ be a one-parametric topological group and $\phi:\IR\to G$ be its parametrization. The map $\phi$ is a topological isomorphism if and only if for some nonempty open set $U\subseteq G$ the preimage $\phi^{-1}[U]$ is bounded in the real line $\IR$.
\end{proposition}

\begin{proof} The ``only if'' part is triavial. To prove the ``if'' part, assume that for some nonempty set $U\subseteq G$, the preimage $\phi^{-1}[U]$ is bounded in the real line.  Replacing the set $U$ by its shift $U-u$, where $u\in U$ is any point, we lose no generality assuming that the identity element $0$ of the group $U$ belongs to the open set $U$. Since all nontrivial subgroups of the real line are unbounded, the surjective homomorphism $\phi:\IR\to G$ is injective and hence is bijective.

We will show that $\phi$ is open, that is for all $\e > 0$ the image $\phi[(-\e,\e)]$ is open in $G$. Replacing $\e$ by a smaller positive number, we can assume that $\phi((-\e, \e)) \subseteq U$, by the continuity of $\phi$. Choose any positive real number $a$ such that the bounded set $\phi^{-1}[U]$ is contained in the closed interval $K\defeq[-a,a]\subset \IR$. The compactness of $K$ ensures that the restriction  $\phi\!\restriction_K:K\to\phi[K]\subseteq G$ is a homeomorphism. Hence, $\phi[(-\e;\e)]$ is open  in the subspace $\phi[K] \subseteq G$. Consequently, there exists an open set $W \subseteq G$  such that $\phi[(-\e;\e)] = \phi[K] \cap W$. Since $\phi[(-\e;\e)] \cap U = \phi[(-\e;\e)]$ and $U\subseteq\phi[K]$, we have $\phi[(-\e;\e)] = \phi[K] \cap W \cap U = W \cap U $. The subsets $W$ and $U$ are open in $G$, hence $\phi[(-\e;\e)]$ is also open in $G$.
Therefore, the bijective continuous homomorphism $\phi:\IR\to G$ is open and hence is a topological isomorphism of the topological groups $\IR$ and $G$.
\end{proof}

The following characterization of the real line is the main result of this section.

\begin{theorem}\label{BMT:1}
A topological group $G$ is isomorphic to $\mathbb{R}$ if and only if  $G$ is one-parametric, metrizable, and not monothetic.
\end{theorem}

\begin{proof}
	The necessity follows directly from the fact that the real line $\IR$ is metrizable, one-parametric, and not monothetic, since for all $x \in \IR$ the subgroup $\langle x \rangle = \IZ x \not = \IR$ is closed. Now we prove the sufficiency.
	
Assume that a one-parametric topological group $G$ is metrizable and not monothetic, and let $\phi: \IR \to G$ be its parametrization.
	
\begin{claim}\label{cl:bounded} There exists a non-empty open set $U\subseteq G$ with bounded preimage $\phi^{-1}[U]\subseteq\IR$.
\end{claim}

\begin{proof} To derive a contradiction, assume that every nonempty open set $U\subseteq G$ has unbounded preimage $\phi^{-1}[U]$ in the real line $\IR$.
Observe that the topological group $G$ is separable, being a continuous image of the separable space $\IR$. Being separable and metrizable, the topology of $G$ has a countable base $\{ B_n\}_{n \in \mathbb{N}}$ consisting of nonempty open sets.

%	 Obeserve that $\phi$ is injective, since otherwise Lemma \ref{BML:1} ensures that $G$ is monothetic. Assume that for all open $U \subseteq G$ the preimage $\phi^{-1}(U)$ is unbounded in $\IR.$ We aim to prove that in this situation $G$ must also be monothetic.
	
 For each $n \in \mathbb{N}$, consider a subset $A_n = \{a \in \IR: \phi[\IZ\cdot a]\cap B_n \not = \emptyset\}$ of $\IR$. We claim that the set $A_n$ is open in the real line $\IR$.
	
	Let $a \in A_n$. We must prove that there exists a real $\e > 0$ such that the interval $(a-\e, a + \e) \subseteq A_n.$ Since $a \in A_n$, there exists $m \in \IZ$ such that $\phi(ma) \in B_n$ and hence, $ma \in \phi^{-1}[B_n].$ Since the set $\phi^{-1}[B_n]$ is open in $\IR$, there exists a real $\delta > 0$ such that the interval $(ma - \delta, ma + \delta) \subseteq \phi^{-1}[B_n]$. Consider the positive real number $\e := | \frac{\delta}{m} |$ and the interval
	$$m \cdot (a - \e, a + \e) = m \cdot (a - |\tfrac{\delta}{m}|, a + |\tfrac{\delta}{m}|) = \{mx: x \in (a - |\tfrac{\delta}{m}|, a + |\tfrac{\delta}{m }|)\} = (ma - \delta, ma + \delta) \subseteq \phi^{-1}(\beta_n). $$
	Hence for all $x \in (a - \e, a + \e)$ we have $ m \phi(x) = \phi(mx) \in B_n$. Therefore, $(a-\e, a+\e) \subseteq A_n$ and $A_n$ is open in $\IR$.
	
	Now we prove that $A_n$ is dense in $\IR$. Given any $x \in \IR\setminus\{0\}$ and positive $\e<|x|$, it suffices to check that $(x - \e, x + \e) \cap A_n\ne \emptyset$. Consider the intervals $m \cdot (x - \e, x + \e) = (mx - |m|\e, mx + |m|\e)$ and observe that for $m\in\IZ$ with $|m|>|x|/\e$ the inequalities $mx + |m|\e > (m+1)x - |m+1|\e$ and $(m-1)x+|m-1|\e>mx-|m|\e$ holds, which means that consecutive intervals $m\cdot(x-\e,x+\e)$ and $(m\pm 1){\cdot}(x-\e,x+\e)$ overlap, covering all points $y \in E := \IR \setminus (|x|/\e)\cdot[x-\e,x+\e]$. Since $\phi^{-1}[B_n]$ is unbounded in $\IR$, we can find $y \in E \cap \phi^{-1}[B_n]$. Hence $y \in (kx - |k|\e, kx + |k|\e) \cap \phi^{-1}[B_n]$ for some $k \in \IZ.$ Therefore, $y = kb$ for some $b \in (x - \e, x + \e)$ and $kb \in \phi^{-1}[B_n]$. This implies $k\phi(b) = \phi(kb) \in B_n$. Hence, $b \in (x - \e, x + \e) \cap A_n$, proving $A_n$ is dense in $\IR.$
	
	Finally, the intersection $A:= \bigcap_{n \in \IN} A_n$ of countably many dense open sets $A_n \subseteq \IR$ is non-empty and, in fact, dense in $\IR$ by Baire Category Theorem. Observe that for all $x \in A$ we have $\phi[\IZ\cdot x]  \cap B_n \not = \emptyset$ for all $n \in \IN$, which implies  $\overline{\phi[\IZ\cdot x]} = G$. Therefore, $G$ is monothetic. This is a contradiction, completing the proof of Claim~\ref{cl:bounded}. 
	\end{proof}
	
By Claim~\ref{cl:bounded}, some nonempty open set $U \subseteq G$ has bounded preimage $\phi^{-1}[U]$ in $\IR$. By Proposition~\ref{p:1-par-bounded=>R}, the parametrization $\phi:\IR\to G$ is a topological isomorphism, witnessing that the one-parametric group $G$ is isomorphic to the real line.
\end{proof}

The proof of Theorem~\ref{BMT:1} yields a bit more, namely:\

\begin{proposition}
If a one-parametric metrizable topological group $G$ is not isomorphic to $\IR$, then $G$ has dense set of topological generators.
\end{proposition}

Theorem \ref{BMT:1} can be rewritten as a characterization of monotheticity in topological groups in the following way.
\begin{corollary}
	A metrizable one-parametric topological group $G$ is monothetic if and only if $G$ is not isomorphic to $\IR$.
\end{corollary}
Metrizability is an essential condition in Theorem \ref{BMT:1}. If it is removed, one obtains the following example of a non-monothetic one-parametric topological group, which is not isomorphic to $\IR$.

\begin{definition} For a topological abelian group $(G,\Tau)$, let $\widehat G$ be the group of all continuous homomorphisms $G\to \IT$ to the circle group $\IT=\{z\in\IC:|z|=1\}$. Elements of the set $\widehat G$ are called {\em characters} on $G$, see \cite{Morris}. The \emph{Bohr topology} on a topological abelian group $(G,\Tau)$ is the smallest topology on $G$ in which all characters $\chi\in\widehat G$ remain continuous. This topology is generated by the subbase consisting of the preimages $\chi^{-1}[U]$ of open sets $U\subseteq\IT$ under characters $\chi\in \widehat G$. By $G^\flat$ we denote the group $G$ equipped with the Bohr topology $\mathcal{T}^\flat$.
\end{definition}

\begin{example}\label{BME:1} \em
	 In particular, the Bohr topology $\Tau_\e^\flat$ of the real line $\IR$ is the smallest topology in which every character
	$$\chi_t(x): \IR \to \mathbb{T}, \quad \chi_t: x \mapsto e^{2\pi itx}$$
	remains continuous. It is known that the topological grup $\IR^{\flat}\defeq(\IR,\Tau_\e^\flat)$ is not metrizable, and therefore is not topologically isomorphic to the real line $(\IR,\Tau_\e)$. Moreover, $\IR^\flat$ is not monothetic. Indeed, for every $a \in \IR\setminus\{0\}$, the continuous character $\chi: \IR \to \mathbb{T}$, $\chi: x \mapsto e^{2\pi ix/a}$, maps the cyclic subgroup $\IZ a$ onto the singleton $\{1\}$, which implies that $\IZ a$ is not dense in $\IR^{\flat}$ for all $a \in \IR^\flat$.
\end{example}

This example motivates the following problems.

\begin{problem}\label{prob1} Find a characterization of the topological group $\IR^\flat$ (among one-parametric topological groups).
\end{problem}

\begin{problem}\label{prob2} Find a characterization of monothetic topological groups among non-metrizable one-parametric topological groups.
\end{problem}

We recall that the {\em weight} $w(X)$ of a topological space $X$ is the smallest cardinarity of a base of the topology of the space $X$. It can be shown that a one-parametric (more generally, separable) topological group is metrizable if and only if it has countable weight. It is known (and easy to see) that the Bohr topology $\Tau_\e^\flat$ on the real line $\IR$ has weight $\mathfrak c$. 

\begin{problem}\label{prob3} Is every non-metrizable one-parametric topological group $G$ of weight $w(G)<\mathfrak c$ monothetic?
\end{problem}

We shall give a partial answer to Problem~\ref{prob3} in Section~\ref{s:3c}.
Problems~\ref{prob1} and \ref{prob2} will be answered in Section~\ref{s:6}, following some preliminary work developed in the Section \ref{s:5}.

\section{The monotheticity of one-parametric topological groups of small weight}\label{s:3c}

In this section we shall provide a partial answer to Problem~\ref{prob3}. 
We recall that a subset $A$ of a topological space $X$ is {\em meager} if $A$ is the countable union of nowhere dense sets in $X$. Let $\mathcal M$ be the family of all meager subsets of the real line. The Baire Theorem ensures that $\IR\notin\mathcal M$. Let $\cov(\mathcal M)$ be the smallest cardinality of a subfamily $\mathcal A\subseteq\M$ such that $\bigcup\A=\IR$. The Baire Theorem ensures that $\cov(\M)\ge\w_1$, where $\w_1$ is the smallest uncountable cardinal. Since $\IR$ is the union of continuum many singletons, $\cov(\M)\le\mathfrak c$. The exact position of the cardinal $\cov(\M)$ in the interval $[\w_1,\mathfrak c]$ is not determined by the ZFC axioms. Martin's Axiom implies $\cov(\M)=\mathfrak c$, but the strict inequality $\cov(\M)<\mathfrak c$ is also consistent, see \cite{BaJu}, \cite{Blass}, \cite{Vaughan}.  

\begin{theorem}\label{t:covM} Every non-metrizable one-parametric topological group $G$ of weight $w(G)<\cov(\M)$ is monothetic.
\end{theorem}

\begin{proof} Let $G$ be a non-metrizable one-parametric topological group of weight $w(G)<\cov(\M)$ and $\phi:\IR\to G$ be its parametrization. By the definition of the cardinal $\kappa\defeq w(G)$, there exists a base $\{B_\alpha\}_{\alpha\in\kappa}$ of the topology of the group $G$, consisting of nonempty open sets in $G$. Since the topological group $G$ is not metrizable, the homomorphism $\phi$ is not a topological isomorphism. Applying Proposition~\ref{p:1-par-bounded=>R}, we conclude that for every $\alpha\in \kappa$, the set $\phi^{-1}[B_\alpha]$ is unbounded in the real line. 
Repeating the argument of the proof of Claim~\ref{cl:bounded}, we can prove that for every $\alpha\in\kappa$, the set $D_\alpha\defeq\{x\in \IR:\IZ{\cdot}x\cap\phi^{-1}[B_\alpha]\ne\varnothing\}$ is open and dense in $\IR$. Then its completement $\IR\setminus D_\alpha$ is closed and nowhere dense in $\IR$. Since $\kappa<\cov(\M)$, the family $\{\IR\setminus D_\alpha\}_{\alpha\in\kappa}$ does not cover the real line. Then there exists a real number $x\in \IR\setminus\bigcup_{\alpha\in\kappa}(\IR\setminus D_\alpha)=\bigcap_{\alpha\in\kappa}D_\alpha$.  For this real number $x$ and its image $y\defeq \phi(x)\in G$, we obtain that $\IZ{\cdot}x\cap\phi^{-1}[B_\alpha]\ne\emptyset$ and hence $\IZ{\cdot}y\cap B_\alpha$ for all $\alpha$. Since $\{B_\alpha\}_{\alpha\in\kappa}$ is a base of the topology of the group $G$, the cyclic subgroup $\IZ{\cdot}y$ is dense in the topological group $G$, witnessing that $G$ is monothetic. 
\end{proof} 

Theorems~\ref{BMT:1} and \ref{t:covM} imply another characterization of the real line in the class of one-parametric topological groups.

\begin{theorem}\label{t:R<=>covM} A topological group $G$ is topologically isomorphic to the real line if and only if $G$ is a non-monothetic one-parametric topological group of weight $w(G)<\cov(\M)$.
\end{theorem}

\begin{problem}  Can the cardinal $\cov(\M)$ in Theorems~\ref{t:covM} and \ref{t:R<=>covM} be replaced by the cardinal $\mathfrak c$?
\end{problem}

\section{The $\IN$-scaled topological groups}\label{s:5}

In this section we define and study $\IN$-scaled topological groups.

\begin{definition} A topological group $(G,+,\Tau)$ is called {\em $\IN$-scaled} if for every $n\in\IN$, the map $G\to G$, $x\mapsto n\cdot x$, is open. 
\end{definition}

\begin{example} The topological groups $\IR,\IR^\flat$, $\IT$ are $\IN$-scaled.
\end{example}

There exists a simple construction (called $\IN$-scaling) transforming any abelian topological group into an $\IN$-scaled topological group. To desribe this construction, we will use the following characterization of an open base at identity, which can be found in \cite[1.3.12]{AT}.

\begin{theorem}\label{ATT}
	Let $G$ be a topological group. A family $\mathcal{U}$ of open subsets of $G$ is a base of the topology of $G$ at the identity $0 \in G$ if and only if the following conditions hold:
	\begin{enumerate}
		\item for every $U \in \mathcal{U}$, there exists $V \in \mathcal{U}$ such that $V + V \subseteq U$;
		\item for every $U \in \mathcal{U}$, there exists $V \in \mathcal{U}$ such that $-V \subseteq U$;
		\item for every $U \in \mathcal{U}$ and $x \in U$, there exists $V \in \mathcal{U}$ such that $V + x \subseteq U$;
		\item for every $U \in \mathcal{U}$ and $x \in G$, there exists $V \in \mathcal{U}$ such that $x + V - x \subseteq U$;
		\item for every $U,V \in \mathcal{U}$, there exists $W \in \mathcal{U}$ such that $W \subseteq U \cap V$;
		\item  $\{0\} = \bigcap \mathcal{U}.$
	\end{enumerate}
Moreover, if $\mathcal U$ is a family of subsets of a group $G$ that satisfies the conditions \textup{(1)--(6)}, then the family $\mathcal{B}_\mathcal{U} = \{g + U: g \in G, U \in \mathcal{U}\}$ is a base for a $T_1$ group topology on $G$. 
\end{theorem}

We recall that for a topological group $(G,+,0,\Tau)$, we denote by $\Tau_0$ the family of all open neighborhoods of the identity element $0$ of the group $G$.

\begin{proposition}
Let $(G, +, 0, \mathcal{T})$ be a topological abelian group. The family $\{g + nU: g \in G, n \in \IN,\; U \in \mathcal{T}_0\}$ is a base for a Hausdorff group topology on $G$.
\end{proposition}

\begin{proof}
	It suffices to prove that the family $\mathcal{U} = \{nU: n \in \IN,\; U \in \mathcal{T}_0\}$ satisfies the conditions from Theorem \ref{ATT}, given that $\mathcal{T}_0$ already satisfies them by the same theorem. We verify the conditions in order:
	\begin{enumerate}
		\item Take arbitrary $nU \in \mathcal{U}$. Since $U \in \mathcal{T}_0$, there exists $V \in \mathcal{T}_0$ with $V + V \subseteq U$. Hence for $nV \in \mathcal{U}$ we have $nV + nV \subseteq nU$ (by the commutativity of the group operation);
		\item Take arbitrary $nU \in \mathcal{U}$. Since $U \in \mathcal{T}_0$, there exists $V \in \mathcal{T}_0$ with $-V \subseteq U$. Hence for $nV \in \mathcal{U}$ we have $-nV \subseteq nU.$
		\item Take arbitrary $nU \in \mathcal{U}$ and $x \in nU$. Then there exists $y \in U$ such that $ny = x$. Since $U \in \mathcal{T}_0$ and $y \in U$, there exists $V \in \mathcal{T}_0$ with $V + y \subseteq U$. Hence for $nV \in \mathcal{U}$ we have $nV + ny = nV + x \subseteq nU.$
		\item Take arbitrary $nU \in \mathcal{U}$ and $x \in G$. Since $G$ is abelian we get $x + nU - x = nU \subseteq nU.$
		\item Take arbitrary $nU, mV \in \mathcal{U}$. Since $U,V \in \mathcal{T}_0$ and the map $\mu_n: G \to G$, $\mu_n: x \mapsto nx$, is continuous for all $n \in \IN$, we can find $W_1, W_2 \in \mathcal{T}_0$ such that $mW_1 \subseteq U$ and $nW_2 \subseteq V$. Put $W:= W_1 \cap W_2$. Then $nmW \subseteq nU$ and $nmW \subseteq mV$. Hence for $nmW \in \mathcal{U}$ we have $nmW \subseteq nU \cap mV$.
		\item $\{0\} \subseteq \bigcap \mathcal{U} \subseteq \bigcap \mathcal{T}_0 = \{0\}$.
	\end{enumerate}
\end{proof}

%Observe that $G \in \mathcal{T} \subseteq A$. Let $x \in nU \cap mV$ for some $n,m \in \IN$ and $U, V \in \mathcal{T}$. Then there exist $u \in U, v \in V$ such that $nu = mv = x$. Since the map $\mu_n: G \to G, \mu: x \mapsto nx$ is continuous for all $n \in \IZ$, the set $W := \mu_m^{-1}(U) \cap \mu_n^{-1}(V)$ is open. Since $G$ is divisible, there exists $w_u \in \mu_m^{-1}(U)$ such that $mw_u = u$. Then $nmw_u = x = mv$ and since $G$ is uniquely-divisible, we conclude $nw_u = v$. Hence $w_u \in W$ and $W$ is an open set with $x \in nmW \subseteq nU \cap mV$ . This implies $A$ is a base for some topology $\mathcal{T}'$ on $G$.
\begin{definition}
	Let $(G, +, 0, \mathcal{T})$ be a topological abelian group. The group topology $\mathcal{T}'$ generated by the base $\{g + nU: g \in G,\; n \in \IN, U \in \mathcal{T}_0\}$ is called the {\em $\IN$-scaling} of the topology $\mathcal{T}$. By $G_s$ we denote the group $G$ equipped with the topology $\mathcal{T}'$. 
\end{definition}

The definition of the $\IN$-scaling topology $\Tau'$ implies the following useful fact.

\begin{proposition} Let $\Tau_1,\Tau_2$ be two group topologies on an abelian group $G$. If $\Tau_1\subseteq\Tau_2$, then $\Tau_1'\subseteq\Tau_2'$.
\end{proposition}

The associativity of the group operation implies that $n(mx)=(nm)x$ for all $n,m\in\IZ$ and all elements $x$ of a group. This fact implies that the operation of $\IZ$-scaling is idempotent and produces an $\IN$-scaled group topology.

\begin{proposition} For every abelian topological group $(G,\Tau)$, its $\IN$-scaling $G_s\defeq(G,\Tau')$ is an $\IN$-scaled topological group.
\end{proposition}

%\begin{definition}
%	A topological abelian group $(G, +, 0, \mathcal{T})$ is called {\em $\IN$-scaled} if $\mathcal{T}' = \mathcal{T}.$
%\end{definition}

The following proposition shows that certain properties of divisible topological groups are preserved by $\IN$-scaling. Recall that a group $G$ is {\em divisible} if for all $n\in\IN$ and $g\in G$ there exists an element $x\in G$ such that $n{\cdot}x=g$. 

\begin{proposition}\label{p:mon-tb=>Nscaling} If a divisible topological abelian group $G$ is monothetic (or totally bounded), then so is its $\IN$-scaling $G_s$.
\end{proposition}

\begin{proof} Let $(G,+,0,\Tau)$ be a divisible topological abelian group and let $G_s\defeq(G,+,0,\Tau')$ be its $\IN$-scaling.
\smallskip

1. Assuming that $(G,\Tau)$ is monothetic, fix any topological generator $a$ of $G$. We claim that $a \in G$ remains a topological generator for the topological group $G_s$. It suffices to show that $\IZ{\cdot}a\cap (g + nU) \not = \emptyset$ for all $g \in G$, $n \in \IN$ and $U \in \mathcal{T}_0$. Since $G$ is divisible, there exists $g' \in G$ such that $g = ng'$. Then $g + nU = n(g' + U)$, since $G$ is abelian. As $g' + U \in \mathcal{T}$, there exists $m \in \IN$ such that $ma \in g' + U$. Hence, $nma \in \IZ{\cdot} a \cap (g + nU)$, which shows that $a$ is indeed a topological generator of the topological group $G_s$.
\smallskip

2. Assume that $G$ is totally bounded. Given any nonempty open set $W\in\Tau'$, we need to find a finite set $F \subseteq G$ such that $F+W=G$. By the definition of the $\IN$-scaled topology $\Tau'$, there exist $n\in\IN$, $x\in G$ and $U\in\Tau_0$ such that $x+nU\subseteq W$. Since the group $G$ is divisible, there exists an element $x'\in G$ such that $x=nx'$. 
			Since the topological group $(G,\Tau)$ is totally bounded, there exists a finite set $F' \subseteq G$ such that $G = F'+(x'+U)$. Consider the finite set $F\defeq nF'$ in $G$. The divisibility and commutativity of the group $G$ ensure that
			$$G = nG = n(F'+x'+U)=nF'+nx'+nU=F+x+nU\subseteq F+W.$$Hence, $F=nF'$ is a required finite subset of $G$, witnessing that the topological group $G_s$ is totally bounded.
\end{proof}

\section{A characterization of monothetic one-parametric topological groups}\label{s:6}

By Theorem~\ref{BMT:1}, a metrizable one-parametric group is monothetic if and only if it is not isomorphic to the real line. In this section we characterize non-metrizable monothetic one-parametric groups. First we shall characterize monothetic group topologies on the real line. The standard Euclidean topology on the real line is denoted by $\Tau_\e$; by $\Tau_\e^\flat$ we denote the Bohr topology on $\IR$. %Note that it is formulated in terms of a topology on $\IR$ that is coarser than the Euclidean topology $\Tau_\e$. This is because any non-metrizable one-parametric topological group is isomorphic to $\IR$ equipped with the pullback topology of $G$, and this topology is coarser that the Euclidean topology. A formulation purely in terms of the one-parametric group $G$ is presented later in this section.

\begin{theorem}\label{BMT:2}
	A group topology $\mathcal{T} \subseteq \mathcal{T}_\varepsilon$ on $\IR$ is monothetic if and only if its $\IN$-scaling $\mathcal{T}'$ does not contain the Bohr topology $\mathcal{T}_\e^\flat$ of $\IR$.
\end{theorem}
\begin{proof}
	We start by proving the necessity. Suppose $\mathcal{T}$ is monothetic. Then $\mathcal{T}'$ is also monothetic by Proposition \ref{p:mon-tb=>Nscaling}. By Example \ref{BME:1}, the Bohr topology $\Tau_\e^\flat$ on $\IR$ is not monothetic, which implies $\mathcal{T}_\e^\flat \not \subseteq \mathcal{T}'$.
	
	Now we prove the sufficiency. Suppose $\mathcal{T}_\e^\flat \not \subseteq \mathcal{T}'$. By the definition of the Bohr topology, this means there exists $a \in \IR$ such that the character $\chi_a: \IR \to \mathbb{T}$, $\chi_a: x \mapsto e^{2\pi iax}$, is not continuous with respect to $\mathcal{T}'$. Clearly, $a \not = 0$ since the trivial character is continuous for any topology on $\IR$. We claim that $1/a$ is a topological generator of the topological group $(\IR,\Tau)$, meaning that the cyclic subgroup $\IZ/a\defeq\{n/a:n\in\IZ\}$ is $\Tau$-dense in $\IR$. In the opposite case, the closure $H\defeq \overline{\IZ/a}$ of $\IZ/a$ in $(\IR,\Tau)$ is a $\Tau$-closed proper subgroup of the real line. It follows from $\Tau\subseteq\Tau'\subseteq \Tau_\e'=\Tau_\e$ that $H$ is a closed proper subgroup in the topological groups $(\IR,\Tau')$ and $(\IR,\Tau_\e)$, which implies that the group $H$ is cyclic. Then the quotient group $\IR/H$ is isomorphic to the circle group $\IT$ via some topological group isomorphism $h: \IR/H \to \mathbb{T}$. Consider the compostion $p = h \circ q$ of $h$ with the quotient homomorphism $q: \IR \to \IR/H$. The compactness of the group $\IR/H$ implies that the quotient topologies on the group $\IR/H$, genereted by the topologies $\Tau'$ and $\Tau_\e$ coincide. Then $p=h\circ q$ is a $\Tau'$-continuous character on $\IR$. Since $\mathcal{T}' \subseteq \mathcal{T}_\varepsilon$, the character $p$ is $\Tau_\e$-continuous and hence there exists $b \in \IR$ such that $p(x)= \chi_b(x) = e^{2\pi i bx}$ for all $x\in\IR$. Observe that $p[H] = \{1\}$, which implies $e^{2\pi i\frac{b}{a}} = 1$. Hence $\frac{b}{a} \in \IZ$, and therefore $b = na$ for some $n \in \IZ$. Then $p(x) = e^{2\pi inax}$. Since the topological group $(\IR,\Tau')$ is $\IN$-scaled, the map $\IR\to\IR$, $x\mapsto nx$, is open in the topology $\Tau'$ in $\IR$, and hence the map $d_n: \IR \to \IR$, $d_n: x \mapsto x/n$, is continuous with respect to $\mathcal{T}'$. But then the character $\chi_a = p \circ d_n$ is $\Tau'$-continuous, which contradicts the choice of $a$. This contradiction shows that the cyclic group $\IZ/a$ is $\Tau$-dense in $\IR$, witnessing that topological group $(\IR,\Tau)$ is monothetic.
\end{proof}

In the following part we aim to reformulate Theorem \ref{BMT:2} for one-parametric topological groups, without referring to the real line. This will be done with the help of the Borel-Bohr topology.

\begin{definition} A function $f:X\to Y$ between topological spaces is called
\begin{enumerate}
\item {\em Borel} if for every open set $U\subseteq Y$ its preimage $f^{-1}[U]$ is a Borel subset of $X$;
\item {\em $F_\sigma$-measurable} if for every open set $U\subseteq Y$ its preimage $f^{-1}[U]$ is a set of type $F_\sigma$ in $X$.
\end{enumerate}
\end{definition}

We recall that a subset $A$  of a topological space $X$ is {\em of type $F_\sigma$} if $A$ is the contable union of closed sets in $X$. A subset $B$ of a topological space $X$ is {\em Borel} if it belongs to the smallest $\sigma$-algebra that contains all open subsets of $X$.

\begin{proposition}\label{BMP:3}
	Let $G$ be a one-parametric topological group and $\phi: \IR \to G$ be a parametrization of $G$. For a homomorphism $\chi:G\to\IT$, the following conditions are equivalent:
\begin{enumerate}
\item the homomorphism $\chi\circ\phi:\IR\to \IT$ is continuous;
\item the homomorphism $\chi:G\to\IT$ is $F_\sigma$-measurable.
\item the homomorphism $\chi:G\to\IT$ is Borel;
\item the homomorphism $\chi\circ\phi:\IR\to \IT$ is Borel.
\end{enumerate}
\end{proposition}

\begin{proof} $(1)\Ra(2)$ Assume that the homomorphism $h\defeq\chi\circ\phi:\IR\to\IT$ is continuous. Then for every open set $U\subseteq\IT$, the preimage $h^{-1}[U]$ is open and hence $\sigma$-compact in the real line. Since continuous maps preserve $\sigma$-compact sets, the image $\phi[h^{-1}[U]]=\chi^{-1}[U]$ is a $\sigma$-compact subset of the topological group $G$. Since $G$ is Hausdorff, the $\sigma$-compact set $\chi^{-1}[U]$ is of type $F_\sigma$ in $G$, witnessing that the homomorphism $\chi$ is $F_\sigma$-measurable.
\smallskip

The implication $(2)\Ra(3)$ is trivial (because $F_\sigma$-subsets of topological spaces are Borel); the implication $(3)\Ra(4)$ follows immediately from the well-known fact that the composition of two Borel maps is Borel; and the implication $(4)\Ra(1)$ follows from the continuity of Borel homomorphisms between Polish groups, see \cite[9.10+11.5]{Kechris}.
\end{proof}

\begin{definition}
	The {\em Borel-Bohr topology}  on a topological abelian group $(G,\Tau)$ is the smallest topology $\Tau^\natural$ on $G$ in which all Borel homomorphisms from $(G,\Tau)$ to the circle group $\IT$ are continuous. The group $G$ endowed with the Borel-Bohr topology $\Tau^\natural$ will be denoted by $G^\natural$.
\end{definition}

Since every continuous homomorphism is Borel, the Borel-Bohr topology of a topological abelian group contains the Bohr topology of the group. The converse is true for the real line (because every Borel homorphism to the circle group on the real line is continuous). This implies that the Bohr-Borel topology $\Tau_\e^\natural$ on the real line coincides with the Bohr topology $\Tau_\e^\flat$ on $\IR$.

Proposition~\ref{BMP:3} implies the following corollary which shows that the Bohr topology on the real line can be recovered from the Borel-Bohr topology of any one-parametric topological group with injective parametrization. 

\begin{corollary}\label{BMP:4}
	Let $G$ be a one-parametric topological group with an injective parametrization $\phi: \IR \to G$. Then $\phi: \IR^\flat \to G^\natural$ is a topological group isomorphism.
\end{corollary}

%\begin{proof} 
%	Take an arbitrary $U \in \mathcal{T}_b(\IR)$. By definition of the Bohr topology, there exists a character $\chi: \IR \to \mathbb{T}$ and open set $V \subseteq \mathbb{T}$ such that $U = \chi^{-1}[V]$. Factor $\chi$ through $G$ using Proposition \ref{BMP:3} to get $\chi = \psi \circ \phi$, where $\psi: G \to \mathbb{T}$ is a Borel group homomorphism, and therefore continuous with respect to $\mathcal{T}_{bb}(G)$. Hence $\phi[U] = \psi^{-1}[V] \in \mathcal{T}_{bb}(G)$. Thus $\phi$ is an open map.
%	Now take an arbitrary $W \in \mathcal{T}_{bb}(G)$. By definition of the Borel-Bohr topology, there exists a Borel group homomorphism $\psi: G \to \mathbb{T}$ and open set $V \subseteq \mathbb{T}$ such that $W = \psi^{-1}[V]$. By proposition \ref{BMP:3}, the map $\chi: \IR \to \mathbb{T}$, $\chi = \psi \circ \phi$ is continuous with respect to $\mathcal{T}_b(\IR)$.  Hence $\phi^{-1}[W] = \chi^{-1}[V] \in \mathcal{T}_{b}(\IR)$. Thus $\phi$ is continuous.
%\end{proof}

The following theorem provides the promised characterization of monothetic topological groups in the class of non-metrizablle one-parametric topological groups.
%internal formulation of Theorem \ref{BMT:2}, the one that does not refer explicitly to $\IR$. Such a formulation is possible only under the assumption that the topological group $G$ admits an injective parametrization, which is precisely the condition ensuring that the Borel–Bohr topology on $G$ is isomorphic to the Bohr topology on $\IR$. We note, however, that if the parametrization $\phi: \IR \to G$ is not injective, then $G$ is either the trivial topological group or is isomorphic to $\mathbb{T}$; in either case $G$ is monothetic.

\begin{theorem}\label{BMT:3} A non-metrizable one-parametric topological group $(G,\Tau)$ is monothetic if and only if the $\IN$-scaling of its topology does not contain the Borel-Bohr topology $\Tau^\natural$ of $(G,\Tau)$.
\end{theorem}

\begin{proof} Let $(G, +, 0, \mathcal{T})$ be a non-metrizable one-parametric topological group and let $\phi: \IR \to G$ be its parametrization.
Assuming that $\phi$ is not injective, we conclude that its kernel $\ker \phi=\phi^{-1}(1)$ is a non-trivial closed subgroup of the real line, which implies that the quotient group $\IR/\ker \phi$ is compact and metrizable and so is the group $G$, which contradicts our assumption. This contradiction shows that the homomorphism $\phi:\IR\to G$ is injective.
%By Corollary~\ref{BMP:4}, the function $\phi:\IR^\flat\to G^\natural$ is a topological isomorphism. 
Consider the group topology $\Tau_G\defeq\{\phi^{-1}[U]:U\in\Tau\}$ on $\IR$ induced by the continuous bijective homomorphism $\phi:\IR\to G$ and observe that $\phi$ is a topological isomorphism of the topological groups $(\IR,\Tau_G)$ and $(G,\Tau)$. By Proposition~\ref{BMP:3}, a homomorphism $\chi:G\to\IT$ is Borel if and only if the composition $\chi\circ \phi:\IR\to\IT$ is continuous. This implies that a homomorphism $\chi:\IR\to\IT$ is $\Tau_G$-Borel if and only if it is $\Tau_\e$-continuous. Now the definitions of the Bohr and Borel--Bohr topologies ensure that the Bohr topology $\Tau_\e^\flat$ on the real line $\IR$ coincides with the Borel-Bohr topology $\Tau_G^\natural$ on $\IR$. 

Now we are ready to prove the equivalence of the conditions in Theorem~\ref{BMT:3}. If the topological group $(G,\Tau)$ is monothetic, then so is its isomorphic copy $(\IR,\Tau_G)$. By Theorem~\ref{BMT:2}, the $\IN$-scaling $\Tau_G'$ of the topology $\Tau_G$ does not contain the Bohr topology $\Tau^\flat_\e$. Since $\Tau_\e^\flat=\Tau_G^\natural$, the topology $\Tau'_G$ does not contain the Borel-Bohr topology $\Tau_G^\natural$ of the topological group $(G,\Tau)$. Since the topological groups $(\IR,\Tau_G)$ and $(G,\Tau)$ are isomorphic, the $\IN$-scaling $\Tau'$ of the topology $\Tau$ does not contain the Borel--Bohr topology $\Tau^\natural$ of the topological group $(G,\Tau)$. 

Now assume conversely that the $\IN$-scaling $\Tau'$ of the topology $\Tau$ does not contain the Borel-Bohr topology $\Tau^\natural$ of the abelian topological group $(G,\Tau)$. Since the topological groups $(G,\Tau)$ and $(\IR,\Tau_G)$ are isomorphic, the $\IN$-scaling $\Tau_G'$ of the topology $\Tau_G$ does not contain the Borel-Bohr topology $\Tau_G^\natural$ of the topological group $(\IR,\Tau_G)$. Taking into account that $\Tau_G^\natural=\Tau_\e^\flat$, we conclude that the topology $\Tau_G'$ does not contain the Bohr topology $\Tau_\e^\flat$ on the real line $\IR$. Applying Theorem~\ref{BMT:2}, we conclude that the topological group $(\IR,\Tau_G)$ is monothetic and so is its isomorphic copy $(G,\Tau)$.
\end{proof}

Theorem \ref{BMT:3} answers Problem~\ref{prob2}. Now we provide an answer to Problem~\ref{prob1}, presenting a characterization of the topological group $\IR^\flat$.

%The following internal characterization of the Bohr topology is known (see \cite{AT},p. ~634).
%\begin{proposition}\label{BMP:5}
%	The Bohr topology $\mathcal{T}_b(G)$ of the topological abelian group $G$ is the finest totally bounded topology on $G$ coarser than the original one.
%\end{proposition}

%Finally, we obtain the following characterization of $\IR^+$ as a topological group.

\begin{theorem} A topological group is isomorphic to the real line endowed with the Bohr topology if and only if it is one-parametric, totally bounded, $\IN$-scaled and not monothetic. 
\end{theorem}

\begin{proof} The ``only if'' part follows from the known properties of the Bohr topology on the real line: it is totally bounded, $\IN$-scaled and not monothetic.

%	Necessity follows directly from the properties of the Bohr topology on $\IR$, introduced in Proposition \ref{BMP:5}, Proposition \ref{BMP:1} and Example \ref{BME:1}
	
	To prove the ``if'' part, assume that a topological group $(G, +, 0, \mathcal{T})$ is one-parametric, totally bounded, $\IN$-scaled, and not monothetic. Let $\phi: \IR \to G$ be a parametrization of $G$. 
	Consider the closed subgroup $\operatorname{ker}\phi$ of $\IR$. Since closed subgroups of $\IR$ are exactly $\{0\}, \IR$, and $a\IZ$ for $a \in \IR_{>0}$, we examine these three cases. 
	
	If $\operatorname{ker}\phi = a\IZ$, then $G \cong \IR/a\IZ \cong \mathbb{T}$, and therefore $G$ is monothetic, a contradiction. If $\operatorname{ker}\phi = \IR$, then $G = \{0\}$, and therefore $0 \in G$ is a topological generator of $G$, again contradicting the assumption on $G$. Thus we conclude that $\phi$ is injective.
	
	The topological group $(G,\Tau)$ is totally bounded and hence it is not isomorphic to the real line. Taking into account that $(G,\Tau)$ is not monothetic, we can apply Theorem~\ref{BMT:1} and conclude that the one-parametric group $(G,\Tau)$ is not metrizable. Taking into account that $(G,\Tau)$ is not monothetic, we can apply Theorem~\ref{BMT:3} and conclude that the $\IN$-scaling $\Tau'$ of the topology $\Tau$ contains the Borel-Bohr topology $\Tau^\natural$ of $(G,\Tau)$. Since the topological group $(G,\Tau')$ is $\IN$-scaled, $\Tau^\natural\subseteq\Tau'=\Tau$. Consider the topology $\Tau_G\defeq\{\phi^{-1}[U]:U\in\Tau\}$ on the real line and observe that $\phi:\IR\to G$ is a topological isomorphism of the topological groups $(\IR,\Tau_G)$ and $(G,\Tau)$. The total boundedness of the topology $\Tau$ implies the total boundedness of the topology $\Tau_G$. The continuity of the homomorphism $\phi:\IR\to G$ implies $\Tau_G\subseteq\Tau_\e$. By \cite[p.~634]{AT}, the totally bounded topology $\Tau_G$ is contained in the Bohr topology $\Tau_\e^\flat$. This means that the map $\phi:\IR^\flat\to G$ is continuous. On the other hand, Corollary~\ref{BMP:4} ensures that $\phi:\IR^\flat\to G^\natural$ is a topological isomorphism, which implies that the identity map $G^\natural\to G$ is continuous and hence $\Tau\subseteq\Tau^\natural$. Taking into account that $\Tau^\natural\subseteq\Tau$, we conclude that $\Tau^\natural=\Tau$. Then $\phi:\IR^\flat\to G$ is a topological isomorphism of the topological groups $\IR^\flat=(\IR,\Tau_\e^\flat)$ and  $(G,\Tau)$.  	
\end{proof}

\end{document}